\documentclass{article}
\usepackage[left=30truemm, right=30truemm]{geometry}
\usepackage{amsmath,amssymb,amsthm}
\usepackage{bm}
\usepackage{mathtools}
\usepackage{xcolor}

\usepackage[shortlabels]{enumitem}
\usepackage{float}

\usepackage{tikz}
\usepackage{caption}
\usepackage{subcaption}
\usetikzlibrary{arrows.meta,calc,decorations.markings,math,arrows.meta}

\usepackage[normalem]{ulem}

\newtheorem{theorem}{Theorem}[]
\newtheorem{proposition}[theorem]{Proposition}

\newtheorem{claim}{Claim}
\newtheorem{conjecture}[theorem]{Conjecture}

\usepackage[pdfauthor={derajan},pdfstartview=XYZ,bookmarks=true,
colorlinks=true,linkcolor=blue,urlcolor=magenta,citecolor=blue,pdftex,bookmarks=true,linktocpage=true,hyperindex=true]{hyperref}

\numberwithin{equation}{section}

\title{Remarks on proper conflict-free degree-choosability of graphs with prescribed degeneracy}
\author{
    Masaki Kashima\thanks{Faculty of Science and Technology, Keio University, Yokohama, Japan. Email: masaki.kashima10@gmail.com} \quad  
	Riste \v{S}krekovski\thanks{Faculty of Mathematics and Physics, University of Ljubljana, Faculty of Information Studies in Novo Mesto, and Rudolfovo - Science and Technology Centre Novo Mesto, Slovenia. Email:skrekovski@gmail.com} \quad 
	Rongxing Xu\thanks{School of Mathematical and Science, Zhejiang Normal University, Jinhua, China. Email:xurongxing@zjnu.edu.cn}}

\begin{document}

\maketitle

\begin{abstract}
    A proper coloring $\phi$ of $G$ is called a proper conflict-free coloring of $G$ if for every non-isolated vertex $v$ of $G$, there is a color $c$ such that $|\phi^{-1}(c)\cap N_G(v)|=1$.
    As an analogy of degree-choosability of graphs, we introduced the notion of proper conflict-free $({\rm degree}+k)$-choosability of graphs.
    For a non-negative integer $k$, a graph $G$ is proper conflict-free $({\rm degree}+k)$-choosable if for any list assignment $L$ of $G$ with $|L(v)|\geq d_G(v)+k$ for every vertex $v\in V(G)$, $G$ admits a proper conflict-free coloring $\phi$ such that $\phi(v)\in L(v)$ for every vertex $v\in V(G)$.
    In this note, we first remark if a graph $G$ is $d$-degenerate, then $G$ is proper conflict-free $({\rm degree}+d+1)$-choosable.
    Furthermore, when $d=1$, we can reduce the number of colors by showing that every tree is proper conflict-free $({\rm degree}+1)$-choosable.
    This motivates us to state a question.
\end{abstract}

\section{Introduction}\label{sec:intro}

Throughout the paper, we only consider simple, finite, and undirected graphs.
Let $\mathbb{N}$ be the set of positive integers.
For a positive integer $k$, let $[k]$ denote the set of integers $\{1,2,\dots , k\}$.

For a graph $G$, a mapping $\phi$ from $V(G)$ to $\mathbb{N}$ is called a \emph{proper coloring of $G$} if $\phi(u)\neq \phi(v)$ for every edge $uv\in E(G)$.
A proper coloring of a graph $G$ in which every vertex of $G$ maps to an integer in $[k]$ is called a proper $k$-coloring of $G$.

Recently, Fabrici, Lu\v{z}ar, Rindo\v{s}ov\'{a}, and Sot\'{a}k~\cite{FLRS2023} introduced a new variation of coloring named proper conflict-free coloring of graphs.
For a graph $G$, a mapping $\phi$ from $V(G)$ to $\mathbb{N}$ is called a \emph{proper conflict-free coloring of $G$} if $\phi$ is a proper coloring of $G$ and every non-isolated vertex $v\in V(G)$ has a color $c$ such that $|\phi^{-1}(c)\cap N_G(v)|=1$, where $N_G(v)$ is the (open) neighborhood of $v$.
A proper conflict-free coloring of a graph $G$ such that every vertex of $G$ maps to an integer in $[k]$ is called a \emph{proper conflict-free $k$-coloring of $G$}.
For a (partial) coloring $\phi$ of $G$ and a vertex $v\in V(G)$, let $\mathcal{U}_{\phi}(v,G)$ denote the set of colors that appear exactly once in the neighborhood of $v$.
Using this notation, a proper conflict-free coloring $\phi$ of $G$ is a proper coloring $\phi$ of $G$ such that $\mathcal{U}_{\phi}(v,G)\neq \emptyset$ for every non-isolated vertex $v\in V(G)$.
The \emph{proper conflict-free chromatic number} of a graph $G$, denoted by $\chi_{\text{pcf}}(G)$, is the least integer $k$ such that $G$ admits a proper conflict-free $k$-coloring.

One major problem in proper conflict-free coloring is the following Brooks-type conjecture, which was posed by Caro, Petru\v{s}evski, and \v{S}krekovski~\cite{CPS2023}.

\begin{conjecture}\label{conj:brooks}
    For every graph $G$ with the maximum degree $\Delta\geq 3$, $\chi_{{\rm pcf}}(G)\leq \Delta+1$.
\end{conjecture}

This conjecture is widely open except for the case $\Delta=3$ by Liu and Yu~\cite{LY2013} and some asymptotic results in the literature~\cite{CCKP2025,CLarxiv,KP2024,L2024,LRarxiv}. 

It is well known that the original Brooks' theorem was generalized to degree-choosability of graphs in Borodin~\cite{B1977} and Erd\H{o}s, Rubin, and Taylor~\cite{ERT1979}.
Following the same direction, we introduced the concept of proper conflict-free $({\rm degree}+k)$-choosability of graphs in \cite{KSXdegree}, as follows.

A list assignment $L$ of a graph $G$ maps each vertex of $G$ to a set of integers.
For a mapping $f$ from $V(G)$ to positive integers, an $f$-list assignment of $G$ is a list assignment $L$ of $G$ with $|L(v)|\geq f(v)$ for every vertex $v\in V(G)$.
In particular, if $f$ is the constant map from $V(G)$ to a positive integer $k$, an $f$-list assignment of $G$ is called a $k$-list assignment of $G$.

For a given graph $G$ and a list assignment $L$ of $G$, a \emph{proper conflict-free $L$-coloring} of $G$ is a proper conflict-free coloring $\phi$ of $G$ such that $\phi(v)\in L(v)$ for every vertex $v\in V(G)$.
For a non-negative integer $k$, a graph $G$ is \emph{proper conflict-free $({\rm degree}+k)$-choosable} if $G$ admits a proper conflict-free $L$-coloring for any $f$-list assignment of $G$, where $f(v)=d_G(v)+k$ for every vertex $v\in V(G)$.
It is natural to ask whether there is an absolute constant $k$  such that every graph is proper conflict-free $({\rm degree}+k)$-choosable, but in fact, even the existence of a constant $k$ such that $\chi_{\text{pcf}}(G)\leq \Delta(G)+k$ for every graph $G$ is unknown.

In this note, we focus on the proper conflict-free $({\rm degree}+k)$-choosability of graphs with a given degeneracy.
For a positive integer $d$, a graph $G$ is $d$-degenerate if every subgraph $H$ of $G$ satisfies $\delta(H)\leq d$.
We first remark the following simple bound, which states the relationship between the degeneracy and proper conflict-free $({\rm degree}+k)$-choosability of graphs.

\begin{proposition}\label{prop:degeneracy}
    If $G$ is a $d$-degenerate graph for some positive integer $d$,
    then $G$ is proper conflict-free $({\rm degree}+d+1)$-choosable.
\end{proposition}

\begin{proof}
    Let $G$ be a $d$-degenerate graph of order $n$. Let $L$ be a list assignment of $G$ satisfying $|L(v)|\geq d_G(v)+d+1$ for each vertex $v\in V(G)$. 
    Since $G$ is $d$-degenerate, we label the vertices of $G$ as $v_1, v_2, \dots, v_n$ so that each vertex has at most $d$ neighbors with smaller indices. 

    We color the vertices greedily in the order $v_1, v_2, \dots, v_n$ as follows: For each $i \in \{1, 2, \dots, n\}$, we assign $v_i$ a color from $L(v_i)$ that differs from the colors of all previously indexed vertices that are either adjacent to $v_i$ or are the earliest (i.e., the smallest-indexed) neighbor of a vertex adjacent to $v_i$. Note that at most $d$ colors are forbidden by the previously colored neighbors and at most $d_G(v)$ colors are forbidden by the earliest neighbors of $N_G(v_i)$, and hence at least one color is available for $v_i$.
    It is obvious that the obtained coloring is a proper coloring of $G$. 
    Furthermore, for each vertex of $G$, the color of the least indexed neighbor appears exactly once in the neighbors, and hence we obtain a proper conflict-free $L$ coloring of $G$.
\end{proof}

This improves a bound from Cranston and Liu~\cite{CLarxiv} of a Brooks-type statement.
One may ask whether the bound is best possible.
As we saw previously, the 5-cycle is 2-degenerate and not proper conflict-free $({\rm degree}+2)$-choosable, and hence we cannot improve the bound in Proposition~\ref{prop:degeneracy} in general.
On the other hand, when $d=1$, we show the following result, that states that the bound can be reduced to $({\rm degree}+d)$.

\begin{theorem}\label{thm:tree}
    Every tree is proper conflict-free $({\rm degree}+1)$-choosable.
\end{theorem}

The proof of the above theorem is given in the next section.
The bound $({\rm degree}+1)$ cannot be reduced to $({\rm degree}+0)$.
Indeed, let us consider a star $K_{1,n-1}$ ($n\geq 3$) with the center $v_0$ and leaves $v_1,v_2,\dots ,v_{n-1}$, and let $L$ be a list assignment of $K_{1,n-1}$ such that $|L(v_0)|=n-1$ and $L(v_1)=L(v_2)=\cdots =L(v_{n-1})=\{1\}$.
Then the center $v_0$ must have $n-1$ neighbors with color 1 no matter what list $v_0$ has, and hence $K_{1,n-1}$ is not proper conflict-free $L$-colorable.

Similarly, we can construct a 2-degenerate graph that is not proper conflict-free $({\rm degree}+1)$-choosable.
Let $G$ be a graph obtained by $n$ copies $C_1,C_2,\dots ,C_n$ of the 4-cycle ($n\geq 1$) by identifying the vertices $v_1,v_2,\dots ,v_n$ into a vertex $v$, where $v_i$ is a vertex of $C_i$.
Obviously $G$ is 2-degenerate.
Let $L$ be a list assignment of $G$ such that $L(v)=\{1,2,\dots,2n+1\}$ and each vertex of $V(C_i)\setminus \{v_i\}$ has a list $\{1,2i,2i+1\}$.
Then, by the conflict-free condition of vertices of the $i$th copy, colors $\{1,2i,2i+1\}$ are forbidden for $v$, and hence there is no color left for $v$.
Hence, $G$ is not proper conflict-free $L$-colorable.

Considering Theorem~\ref{thm:tree} and these constructions, we ask whether the following may hold.

\begin{conjecture}\label{conj:degeneracy}
    If $G$ is a $d$-degenerate graph for some positive integer $d$,
    then $G$ is proper conflict-free $({\rm degree}+d)$-choosable.
\end{conjecture}

The above constructions imply that the bound in Conjecture~\ref{conj:degeneracy} is best possible for $d=1,2$.
Also, by Theorem~\ref{thm:tree}, Conjecture~\ref{conj:degeneracy} holds for the case $d=1$.
For other values of $d$, we know that every connected outerplanar graph other than the 5-cycle is 2-degenerate and proper conflict-free $({\rm degree}+2)$-choosable~\cite{KSX2025}, and that every planar graph is 5-degenerate and proper conflict-free $({\rm degree}+5)$-choosable~\cite{KSX2025+}, which will appear in separate papers.

\section{Proof of Theorem~\ref{thm:tree}}\label{section:tree}

Suppose that the statement is false, and let $T$ be a minimum counterexample.
Obviously, $|V(T)|\geq 3$.
Let $L$ be a list assignment of $T$ such that $|L(v)|=d_T(v)+1$ for every vertex $v\in V(T)$ and $T$ is not proper conflict-free $L$-colorable.

\begin{claim}\label{claim:2path}
    Let $v_1v_2v_3$ be a path of length 2 of $T$ with $d_T(v_1)=1$ and $d_T(v_2)=2$.
    Then $L(v_1)\subseteq L(v_2)$.
\end{claim}

\begin{proof}
    Assume to the contrary that $L(v_1)\setminus L(v_2)\neq \emptyset$.
    Take a color $\alpha\in L(v_1)\setminus L(v_2)$.
    Let $T'=T-\{v_1,v_2\}$ and let $L'$ be a list assignment of $T'$ defined by $L'(v_3)=L(v_3)\setminus \{\alpha\}$ and $L'(v)=L(v)$ for every $v\in V(T')\setminus\{v_3\}$.
    Note that $|L'(v)|\geq d_{T'}(v)+1$ for every $v\in V(T')$.
    By the minimality of $T$, $T'$ admits a proper conflict-free $L'$-coloring $\phi$.
    Let $\phi(v_3)=\beta\neq \alpha$ and let $\gamma$ be a color in $\mathcal{U}_{\phi}(v_3,T')$.
    By setting $\phi(v_1)=\alpha$ and choosing $\phi(v_2)\in L(v_2)\setminus\{\beta,\gamma\}$, $\phi$ can be extended to a proper conflict-free $L$-coloring of $T$, a contradiction.
\end{proof}

\begin{claim}\label{claim:3path}
    $T$ does not have a path $v_1v_2v_3v_4$ of length 3 with $d_T(v_1)=1$ and $d_T(v_2)=d_T(v_3)=2$.
\end{claim}

\begin{proof}
    Assume to the contrary that $T$ has such a path $v_1v_2v_3v_4$.
    By Claim~\ref{claim:2path}, we know that $L(v_1)\subseteq L(v_2)$.
    As $|L(v_1)|=2$ and $|L(v_2)|=3$, let $\alpha$ be a color in $L(v_2)\setminus L(v_1)$.
    Let $T'=T-\{v_1,v_2,v_3\}$ and let $L'$ be a list assignment of $T'$ defined by $L'(v_4)=L(v_4)\setminus\{\alpha\}$ and $L'(v)=L(v)$ for every $v\in V(T')\setminus\{v_4\}$.
    By the minimality of $T$, $T'$ admits a proper conflict-free $L'$-coloring $\phi$.
    Let $\phi(v_4)=\beta\neq \alpha$ and let $\gamma$ be a color in $\mathcal{U}_{\phi}(v_4,T')$. Note that it is possible that $\gamma=\alpha$.
    We choose $\phi(v_3)\in L(v_3)\setminus\{\beta,\gamma\}$.
    We let $\phi(v_2)=\alpha$ if $\phi(v_3)\neq\alpha$, and let $\phi(v_2)$ be a color in $L(v_2)\setminus\{\phi(v_3),\beta\}$ otherwise.
    In either case, one of $v_2$ and $v_3$ is colored with $\alpha$, which is not in $L(v_1)$.
    Since $|L(v_1)\setminus \{\phi(v_2),\phi(v_3)\}|\geq 1$, we can choose $\phi(v_1)\in L(v_1)\setminus\{\phi(v_2),\phi(v_3)\}$, and hence $\phi$ can be extended to a proper conflict-free $L$-coloring of $T$, a contradiction.
\end{proof}

By Claim~\ref{claim:3path}, $T$ has a vertex $v_0$ of degree at least $3$ such that each component of $T-v_0$ except one component is isomorphic to $K_1$ or $K_2$.
Let $N_T(v_0)=\{x_0, x_1, \dots , x_k\}$.
Note that $k=d_T(v_0)-1\geq 2$.
For each $i\in \{0,1,\dots , k\}$, let $T_i$ denote the component of $T-v_0$ that contains $x_i$.
Without loss of generality, we may assume that $T_i$ is isomorphic to $K_2$ for every $i\in \{1,2,\dots ,\ell\}$ and $T_i$ is isomorphic to $K_1$ for every $i\in \{\ell+1, \ell+2,\dots , k\}$, where $\ell$ is a positive integer at most $k$.
For each $i\in \{1,2,\dots , \ell\}$, let $V(T_i)=\{x_i,y_i\}$ (Figure~\ref{fig:pcf tree}).

By Claim~\ref{claim:2path}, we have $L(y_i)\subseteq L(x_i)$ for each $i\in\{1,2,\dots , \ell\}$.
Thus, we let $L(x_i)=\{\alpha_i,\beta_i,\gamma_i\}$ and $L(y_i)=\{\beta_i,\gamma_i\}$ for each $i\in \{1,2,\dots , \ell\}$, and let $L(x_i)=\{\alpha_i,\beta_i\}$ for each $i\in \{\ell+1,\ell+2,\dots ,k\}$.

In the rest of the proof, we take a proper conflict-free coloring of $T':=T-\left(\bigcup_{i=1}^kV(T_i)\cup \{v_0\}\right)$ and extend it to $T$.

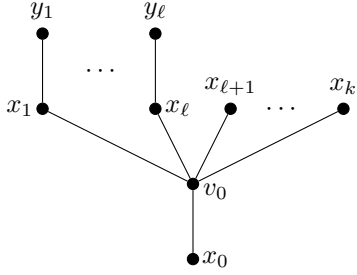
\begin{figure}
    \centering
    \begin{tikzpicture}[roundnode/.style={circle, draw=black,fill=black, minimum size=1.5mm, inner sep=0pt}]
    \node [roundnode] (v0) at (0,0){};
    \node [roundnode] (x0) at (0,-1){};
    \node [roundnode] (x1) at (-2,1){};
    \node [roundnode] (x2) at (-0.5,1){};
    \node [roundnode] (x3) at (0.5,1){};
    \node [roundnode] (x4) at (2,1){};
    \node [roundnode] (y1) at (-2,2){};
    \node [roundnode] (y2) at (-0.5,2){};

    \foreach \i in {0,1,2,3,4} \draw (v0)--(x\i);
    \foreach \i in {1,2} \draw (y\i)--(x\i);

    \node at (0.3,-0.1){$v_0$};
    \node at (0.3,-1){$x_0$};
    \node at (-2.3,1){$x_1$};
    \node at (-0.2,1){$x_\ell$};
    \node at (0.5,1.3){$x_{\ell+1}$};
    \node at (2,1.3){$x_k$};
    \node at (-2,2.3){$y_1$};
    \node at (-0.5,2.3){$y_\ell$};
    \node at (-1.2,1.5){$\cdots$};
    \node at (1.2,1){$\cdots$};
    \end{tikzpicture}
    \caption{A reducible structure of $T$.}
    \label{fig:pcf tree}
\end{figure}

We first consider relatively simple three cases.
In the following three cases, let $\phi$ be a proper conflict-free $L$-coloring of $T'$.
Let $\alpha=\phi(x_0)$ and let $\beta$ be a color in $\mathcal{U}_{\phi}(x_0,T')$.

\medskip
\noindent
\textbf{Case 1.} $k=2$.
\medskip

\noindent
Note that $d_T(v_0)=3$ in this case.
If $\ell=0$, then we let $\phi(x_2)\in L(x_2)\setminus\{\alpha\}$, $\phi(v_0)\in L(v_0)\setminus\{\alpha,\beta,\phi(x_2)\}$, and $\phi(x_1)\in L(x_1)\setminus\{\phi(v_0)\}$.
If $\ell=1$, then we let $\phi(x_2)\in L(x_2)\setminus\{\alpha\}$, $\phi(v_0)\in L(v_0)\setminus\{\alpha,\beta,\phi(x_2)\}$, $\phi(y_1)\in L(y_1)\setminus\{\phi(v_0)\}$, and $\phi(x_1)\in L(x_1)\setminus\{\phi(v_0),\phi(y_1)\}$.
In either case, since $d_T(v_0)=3$ and at least two colors appear in the neighbors of $v_0$, we have $\mathcal{U}_{\phi}(v_0,T)\neq \emptyset$.
Thus, we obtain a proper conflict-free $L$-coloring of $T$, a contradiction.

Now assume that $\ell=2$.
We consider another list assignment $L'$ of $T'$.
If $\alpha_1=\alpha_2$, then let $L'$ be a list assignment of $T'$ defined by $L'(x_0)=L(x_0)\setminus\{\alpha_1\}$ and $L'(v)=L(v)$ for every $v\in V(T')\setminus\{x_0\}$.
Otherwise, let $L'=L$.
Note that $|L'(v)|\geq d_{T'}(v)+1$ for every $v\in V(T')$.
By the minimality of $T$, $T'$ admits a proper conflict-free $L'$-coloring $\phi'$.
Let $\phi'(x_0)=\alpha'$ and let $\beta'$ be a color in $\mathcal{U}_{\phi'}(x_0,T')$.
By the definition of $L'$, either $\alpha_1\neq \alpha'$ or $\alpha_2\neq \alpha'$ holds.
Without loss of generality, we may assume that $\alpha_1\neq \alpha'$.
Then let $\phi'(x_1)=\alpha_1$, $\phi'(v_0)\in L(v_0)\setminus\{\alpha',\beta',\alpha_1\}$, $\phi'(y_i)\in L(y_i)\setminus\{\phi'(v_0)\}$ for $i=1,2$, and $\phi'(x_2)\in L(x_2)\setminus\{\phi'(v_0),\phi'(y_2)\}$.
It is easy to verify that $\phi'$ is a proper conflict-free $L$-coloring of $T$, a contradiction.

\medskip
\noindent
\textbf{Case 2.} $k\geq 3$ and $\ell=k$.
\medskip

\noindent
For each color $c\in L(v_0)\setminus\{\alpha,\beta\}$, 
let $I_c=\{i\mid 1\leq i\leq \ell, c\in L(y_i)\}$.
Since  $|L(v_0)\setminus\{\alpha,\beta\}|\geq k$ and $\sum_{c\in L(v_0)\setminus\{\alpha,\beta\}}|I_c|\leq \sum_{i=1}^\ell |L(y_i)|=2\ell$, there is a color $\gamma\in L(v_0)\setminus\{\alpha,\beta\}$ such that $|I_\gamma|\leq \frac{2\ell}{k}=2$.
Set $\phi(v_0)=\gamma$.
For each $i\in I_\gamma$, let $\phi(y_i)\in L(y_i)\setminus\{\gamma\}$ and let $\phi(x_i)\in L(x_i)\setminus\{\gamma,\phi(y_i)\}$.
Since $|I_\gamma|\leq 2<\ell$, we may assume that $1\notin I_\gamma$.

Now we color remaining vertices.
Since there are at most three colored neighbors of $v_0$ including $x_0$, either (a) there is a color $\alpha'$ that appears exactly once in the colored neighbors of $v_0$, or (b) all neighbors of $v_0$ are colored by $\alpha$.
For each case, we color the neighbors of $v_i$ in the following manner:
\begin{itemize}
    \item If (a), then let $\phi(x_i)\in L(x_i)\setminus\{\gamma,\alpha'\}$ for each $i\in \{1,2,\dots , \ell\}\setminus I_\gamma$.
    \item If (b), then let $\phi(x_1)\in L(x_1)\setminus\{\gamma,\alpha\}$ and let $\phi(x_i)\in L(x_i)\setminus\{\gamma,\phi(x_1)\}$ for each $i\in \{2,3,\dots , \ell\}\setminus I_\gamma$.
\end{itemize}
Then, we have $\alpha'\in \mathcal{U}_{\phi}(v_0,G)$ if (a), and $\phi(x_1)\in \mathcal{U}_{\phi}(v_0,G)$ if (b).
Finally, for each $i\in \{1,2,\dots , \ell\}\setminus I_\gamma$, we choose $\phi(y_i)\in L(y_i)\setminus \{\phi(x_i)\}$.
Since $\gamma\notin L(y_i)$ for each $i\in \{1,2,\dots ,\ell\}\setminus I_{\gamma}$, we know that $\phi(y_i)\neq \gamma=\phi(v_0)$ and hence $\gamma\in \mathcal{U}_{\phi}(x_i,G)$.
Thus, $\phi$ is a proper conflict-free $L$-coloring of $T$, a contradiction.

\medskip
\noindent
\textbf{Case 3.} $k\geq 3$ and $\ell=k-1$.

\medskip
\noindent
We define the color $\gamma$ and the set $I_\gamma$ similarly to Case 2. Note that the colors in $L(x_k)$ are not considered when we define $\gamma$ and $I_\gamma$ in this case.
By the assumption of this case and the choice of the color $\gamma$, we know that $|I_\gamma|\leq \left\lfloor\frac{2\ell}{k}\right\rfloor=1$.
Note that $d_G(x_k)=1$ and $d_G(x_i)=2$ for each $i\leq k-1$.

Set $\phi(v_0)=\gamma$.
We let $\phi(y_i)\in L(y_i)\setminus\{\gamma\}$ and $\phi(x_i)\in L(x_i)\setminus\{\gamma,\phi(y_i)\}$ for $i\in I_\gamma$, and let $\phi(x_k)\in L(x_k)\setminus\{\gamma\}$.
Then, the number of colored neighbors of $v_0$ is equal to $|I_\gamma\cup\{x_0,x_k\}|=|I_\gamma|+2\leq 3$.
Therefore, by a similar argument as in Case 2, we can extend $\phi$ to a proper conflict-free $L$-coloring of $T$, a contradiction.

\medskip
By the above three cases, we may assume that $k\geq 3$ and $\ell\leq k-2$.
We set $X=\{x_1,x_2,\dots ,x_k\}$.
Let $L(X)=\bigcup_{x\in X}L(x)$ and let $\tilde{L}(v_0)=L(v_0)\setminus L(X)$. We consider two cases depending on whether $\bigl|\tilde{L}(v_0)\bigr|\geq 2$ or not.

\medskip
\noindent
\textbf{Case 4.} $\bigl|\tilde{L}(v_0)\bigr|\geq 2$.

\medskip
\noindent
By Claim~\ref{claim:2path}, $\tilde{L}(v_0)\cap L(y_i)=\tilde{L}(v_0)\cap L(x_i)=\emptyset$ for every $i\in \{1,2,\dots , \ell\}$.
We fix a color $\gamma\in \tilde{L}(v_0)$, and let $L'$ be a list assignment of $T'$ defined by $L'(x_0)=L(x_0)\setminus \{\gamma\}$ and $L'(v)=L(v)$ for every $v\in V(T')\setminus\{x_0\}$.
By the minimality of $T$, $T'$ admits a proper conflict-free $L'$-coloring $\phi$.
Let $\phi(x_0)=\alpha$ and let $\beta$ be a color in $\mathcal{U}_{\phi}(x_0,T')$. Note that it is possible that $\beta=\gamma$.

If $\tilde{L}(v_0)\setminus \{\alpha,\beta\}\neq \emptyset$, then let $\phi(v_0)\in \tilde{L}(v_0)\setminus \{\alpha,\beta\}$, let $\phi(x_i)\in L(x_i)\setminus\{\alpha\}$ for each $i\in \{1,2,\dots , k\}$ and let $\phi(y_i)\in L(y_i)\setminus\{\phi(x_i)\}$ for each $i\in\{1,2,\dots , \ell\}$.
This extends $\phi$ to a proper conflict-free $L$-coloring of $T$, a contradiction.

Thus, we infer that $\tilde{L}(v_0)\setminus \{\alpha,\beta\}=\emptyset$, which implies that $\alpha\in \tilde{L}(v_0)$ and $\beta=\gamma$.
Then we choose $\phi(v_0)\in L(v_0)\setminus\{\alpha,\beta\}$ arbitrarily, let $\phi(y_i)\in L(y_i)\setminus\{\phi(v_0)\}$ and let $\phi(x_i)\in L(x_i)\setminus\{\phi(v_0),\phi(y_i)\}$ for each $i\in\{1,2,\dots , \ell\}$, and 
let $\phi(x_i)\in L(x_i)\setminus\{\phi(v_0)\}$ for each $i\in \{\ell+1,\ell+2,\dots , k\}$.
Since $\alpha\notin L(X)$, we have $\alpha\in \mathcal{U}_{\phi}(v_0,T)$ and hence $\phi$ is a proper conflict-free $L$-coloring of $T$, a contradiction.

\medskip
\noindent
\textbf{Case 5.} $\bigl|\tilde{L}(v_0)\bigr|\leq 1$.

\medskip
\noindent
The assumption of this case implies that $|L(X)|\geq |L(v_0)|-1\geq k+1$.
For each color $c\in L(X)$, let $J_c=\{i\mid 1\leq i\leq k, c\in L(x_i)\}$.
Note that $J_c\neq \emptyset$ for any color $c\in L(X)$.
Let $\gamma$ be a color in $L(X)$ such that $|J_\gamma|$ is the smallest among all colors in $L(X)$.
Since $|L(X)|\geq k+1$ and $\sum_{c\in L(X)}|J_c|=\sum_{i=1}^k |L(x_i)|=2k+\ell$, 
we have $|J_\gamma|\leq \left\lfloor\frac{2k+\ell}{k+1}\right\rfloor\leq 2$.
In particular, if $\ell\leq 1$, then $|J_\gamma|=1$.
Let $L'$ be a list assignment of $T'$ defined by $L'(x_0)=L(x_0)\setminus \{\gamma\}$ and $L'(v)=L(v)$ for every $v\in V(T')\setminus \{x_0\}$.
By the minimality of $T$, $T'$ admits a proper conflict-free $L'$-coloring $\phi$.
Let $\alpha=\phi(x_0)$ and let $\beta$ be a color in $\mathcal{U}_{\phi}(x_0,T')$.

\medskip
\noindent
\textbf{Subcase 5.1.} $\ell\leq 1$.

\medskip
\noindent
Let $J_\gamma=\{p\}$.
We set $\phi(x_p)=\gamma$ and let $\phi(y_p)\in L(y_p)\setminus \{\phi(x_p)\}$ if $y_p$ exists.
Since $|L(v_0)|=k+2\geq 5$, we choose $\phi(v_0)\in L(v_0)\setminus\{\alpha,\beta,\gamma,\phi(y_{p})\}$.
For $i\in \{1,2,\dots , \ell\}\setminus\{p\}$, we let $\phi(y_i)\in L(y_i)\setminus \{\phi(v_0)\}$ and let $\phi(x_i)\in L(x_i)\setminus \{\phi(v_0),\phi(y_i)\}$.
For $i\in \{\ell+1, \ell+2,\dots , k\}\setminus \{p\}$, let $\phi(x_i)\in L(x_i)\setminus \{\phi(v_0)\}$.
Since $\gamma\in \mathcal{U}_{\phi}(v_0,T)$, $\phi$ is a proper conflict-free $L$-coloring of $T$, a contradiction.

\medskip
\noindent
\textbf{Subcase 5.2.} $\ell\geq 2$.

\medskip
\noindent
The assumption of the subcase, together with the assumption $\ell\leq k-2$, implies that $k\geq \ell+2\geq 4$.
If there is a color $c\in L(X)$ with $|J_c|\leq 1$, then we argue in a similar way as in the previous Subcase 5.1.
We may now assume that $|J_c|\geq 2$ for every color $c\in L(X)$, and in particular we know that $|J_\gamma|=2$.
Let $J_\gamma=\{p,q\}$ for some $1\leq p<q\leq k$.

Suppose first that $k\geq 5$.
We let $\phi(x_p)=\gamma$, $\phi(x_q)\in L(x_q)\setminus\{\gamma\}$, and for each $i\in\{p,q\}$, let $\phi(y_i)\in L(y_i)\setminus \{\phi(x_i)\}$ if $y_i$ exists.
Since $|L(v_0)|=k+2\geq 7$, we choose $\phi(v_0)\in L(v_0)$ distinct from $\alpha, \beta, \gamma, \phi(x_q)$, and also distinct from $\phi(y_p)$ and $\phi(y_q)$ in case they are defined.
For $i\in \{1,2,\dots , \ell\}\setminus\{p,q\}$, we let $\phi(y_i)\in L(y_i)\setminus \{\phi(v_0)\}$ and let $\phi(x_i)\in L(x_i)\setminus \{\phi(v_0),\phi(y_i)\}$.
For $i\in \{\ell+1, \ell+2,\dots , k\}\setminus \{p,q\}$, let $\phi(x_i)\in L(x_i)\setminus \{\phi(v_0)\}$.
Since $\gamma\in \mathcal{U}_{\phi}(v_0,T)$, $\phi$ is a proper conflict-free $L$-coloring of $T$, a contradiction.

Now we may assume that $k=4$.
Since $\ell\leq k-2=2$ and $2(k+1)\leq \sum_{c\in L(X)}|J_c|=\sum_{i=1}^k|L(x_i)|=2k+\ell$, we infer that $\ell=2$ and $|J_c|=2$ for every $c\in L(X)$.
Thus, without loss of generality, we may assume that $q=4$, which implies that $d_T(x_q)=1$.
Then we let $\phi(x_q)=\gamma$, $\phi(x_p)\in L(x_p)\setminus\{\gamma\}$, and let $\phi(y_p)\in L(y_p)\setminus \{\phi(x_p)\}$ if $y_p$ exists.
Since $|L(v_0)|=k+2=6$, 
we choose $\phi(v_0)\in L(v_0)$ distinct from $\alpha, \beta, \gamma, \phi(x_p)$, and also distinct from $\phi(y_p)$ in case it is defined.
For $i\in \{1,2\}\setminus\{p\}$, we let $\phi(y_i)\in L(y_i)\setminus \{\phi(v_0)\}$ and let $\phi(x_i)\in L(x_i)\setminus \{\phi(v_0),\phi(y_i)\}$.
For $i\in \{3,4\}\setminus \{p,q\}$, let $\phi(x_i)\in L(x_i)\setminus \{\phi(v_0)\}$.
Since $\gamma\in \mathcal{U}_{\phi}(v_0,T)$, $\phi$ is a proper conflict-free $L$-coloring of $T$, a contradiction.

\medskip
This completes the proof of Theorem~\ref{thm:tree}.

\section*{Acknowledgement}

M. Kashima has been supported by JSPS KAKENHI Grant No.~25KJ2077.
R. \v{S}krekovski has been partially supported by the Slovenian Research Agency and Innovation (ARIS) program P1-0383, project J1-3002, and the annual work program of Rudolfovo. 
R. Xu has been supported by National Science Foundation for Young Scientists of China under Grant No. 12401472, and Zhejiang Provincial Natural Science Foundation of China under Grant No. LQN25A010011.

\end{document}